\documentclass[11pt,a4paper]{amsart}
\usepackage[utf8]{inputenc}
\usepackage[T1]{fontenc}
\usepackage[english]{babel}
\usepackage{sidecap}
\usepackage{amscd}
\usepackage{mathscinet}
\usepackage{verbatim}
\usepackage{lmodern}
\usepackage{amsmath}
\usepackage{amssymb} 
\usepackage{amsthm} 
\usepackage{enumitem}
\usepackage{geometry}
\usepackage{graphicx}
\usepackage{hyperref}
\usepackage{mathtools}
\usepackage{bbm}
\usepackage{cite}
\usepackage{indentfirst}
\usepackage{mathtools,amssymb}
\usepackage{tikz}
\usepackage{tikz-cd,mathtools}

\usetikzlibrary{positioning}
\usetikzlibrary{arrows}
\usetikzlibrary{decorations.markings}
\usetikzlibrary{arrows.meta,shapes,decorations.pathreplacing}

\theoremstyle{theorem}
\newtheorem{theorem}[subsection]{Theorem}
\newtheorem*{theorem*}{Theorem \ref{injproj}}
\newtheorem{cor}[subsection]{Corollary}
\newtheorem{prop}[subsection]{Proposition}
\newtheorem{lemme}[subsection]{Lemma}

\theoremstyle{definition}
\newtheorem{defn}[subsection]{Definition}
\newtheorem{ex}[subsection]{Example}
\newtheorem{notation}[subsection]{Notation}

\theoremstyle{remark}
\newtheorem{rem}[subsection]{Remark}

\tikzstyle directed=[postaction={decorate,decoration={markings, mark=at position .5 with {\arrow{angle 90}}}}]
\tikzstyle directed2=[postaction={decorate,decoration={markings, mark=at position .3 with {\arrow{angle 90}}}}]
\tikzstyle directed3=[postaction={decorate,decoration={markings, mark=at position .2 with {\arrow{angle 90}}}}]

\newlist{rome}{enumerate}{7}
\setlist[rome]{label=(\roman*)}

\newenvironment{tz}{\begin{center}\begin{tikzpicture}[scale=1, every node/.style={scale=.9}]}{\end{tikzpicture}\end{center}}

\tikzstyle{b}=[fill=white]

\newcommand{\cev}[1]{\reflectbox{\ensuremath{\vec{\reflectbox{\ensuremath{#1}}}}}}

\newcommand{\ca}{\mathcal{A}}
\newcommand{\cb}{\mathcal{B}}
\newcommand{\cc}{\mathcal{C}}

\newcommand{\cf}{\mathcal{F}}
\newcommand{\cw}{\mathcal{W}}

\newcommand{\cm}{\mathcal{M}}
\newcommand{\cn}{\mathcal{N}}
\newcommand{\ck}{\mathcal{K}}

\newcommand{\Ll}{\cev{\mathcal L}}
\newcommand{\Rl}{\cev{\mathcal R}}
\newcommand{\Lr}{\vec{\mathcal L}}
\newcommand{\Rr}{\vec{\mathcal R}}

\newcommand{\Hom}{\operatorname{Hom}}

\newcommand{\Nu}{\mathcal{V}}

\newcommand{\id}{\operatorname{id}}

\newcommand{\ob}{\operatorname{Ob}}

\newcommand{\sset}{\mathrm{sSet}}

\newcommand{\CAT}{\mathrm{CAT}}

\newcommand{\op}{\operatorname{op}}

\newcommand{\Cyl}{\operatorname{Cyl}}
\newcommand{\Path}{\operatorname{Path}}
\newcommand{\Sp}{\mathrm{Sp}^\Sigma}
\newcommand{\ChR}{\mathrm{Ch}_R}
\newcommand{\Mod}{\mathrm{Mod}}

\newcommand{\ObD}{\operatorname{Ob}D}

\newcommand{\inj}{\operatorname{inj}}
\newcommand{\proj}{\operatorname{proj}}

\newcommand{\lift}{
\ensuremath{\begin{tikzpicture}
\draw[line width=.5pt] (1ex,1ex) -- (2ex,1ex);
\draw[line width=.5pt] (1ex,0) -- (2ex,0);
\draw[line width=.5pt] (1ex,0) -- (1ex,1ex);
\draw[line width=.5pt] (2ex,0) -- (2ex,1ex);
\draw[line width=.5pt] (1ex,0) -- (2ex, 1ex);
\end{tikzpicture}
}}

\begin{document}

\title[Injective and projective model structures on enriched diagrams]{Injective and projective model structures on enriched diagram categories}

\author{Lyne Moser}
\address{UPHESS BMI FSV, Ecole Polytechnique Fédérale de Lausanne, Station 8, CH-1015 $\text{Lausanne}$,
Switzerland}
\email{lyne.moser@epfl.ch}

\date{\today}

\begin{abstract}
In the enriched setting, the notions of injective and projective model structures on a category of enriched diagrams also make sense. In this paper, we prove the existence of these model structures on enriched diagram categories under local presentability, accessibility, and ``acyclicity'' conditions, using the methods of lifting model structures from an adjunction introduced by Garner, Hess, K\polhk edziorek, Riehl, and Shipley. 
\end{abstract}

\maketitle

\section{Introduction}

Categories of diagrams in a model category might be endowed with two particularly useful model structures: the injective and projective ones. In the non-enriched case, the existence of such model structures on categories of diagrams in a combinatorial model category is folklore. However, the  breakthrough by Garner, Hess, K\polhk{e}dziorek, Riehl, and Shipley in \cite{HKRS}, and then in \cite{GKR}, provides new tools which unify the proof of existence of both injective and projective model structures under weaker assumptions, namely for \emph{accessible} model categories, i.e.~locally presentable model categories with accessible functorial factorizations. 
In this paper, we provide an enriched version of this result using the methods of lifting model structures of \cite{HKRS} and \cite{GKR}. We prove the following main theorem. For $(\Nu, \otimes, I)$ a closed symmetric monoidal category, we write~$[D,\ca]_0$ for the category of $\Nu$-functors from a small $\Nu$-category $D$ to a $\text{$\Nu$-category}$~$\ca$. 

\begin{theorem*}
Let $(\Nu, \otimes, I)$ be a locally presentable base (Definition \ref{base}). Suppose~$\ca$~is a $\Nu$-complete locally $\Nu$-presentable $\Nu$-category such that its underlying category $\ca_0$ admits an accessible model structure, and let $D$ be a small $\Nu$-category. 
\begin{rome}
\item If the functors $-\otimes D(d,d')\colon \ca_0\to \ca_0$ preserve cofibrations for all $d,d'\in D$, the injective model structure on the category $[D, \ca]_0$ exists. 
\item If the functors $-\otimes D(d,d')\colon \ca_0\to \ca_0$ preserve trivial cofibrations for all ${d,d'\in D}$, the projective model structure on the category $[D, \ca]_0$ exists. 
\end{rome}
\end{theorem*}

Some results about the existence of injective and projective model structures on categories of enriched diagrams can already be found in the literature. Lurie proves the case of a combinatorial $S$-enriched model category in \cite[Proposition A.3.3.2]{LUR}, where $S$ is an \emph{excellent} model category, e.g.~the category of simplicial sets with the Quillen model structure. On the other hand, in \cite[Theorem 4.2]{DRO}, Dundas, R\"ondigs, and {\O}stv{\ae}r prove the existence of the projective model structure on categories of $\Nu$-diagrams in a symmetric monoidal model category $\Nu$, which is \emph{weakly finitely generated}, i.e.~cofibrantly generated in a stronger sense, and which satisfies the monoid axiom.  They use this result to construct models for stable homotopy categories, which give examples of applications of the homotopy theory of enriched diagrams to the fields of equivariant stable homotopy theory and motivic homotopy theory. 

Categories of enriched diagrams, in particular the category of simplicial functors $\sset_*^{\mathrm{fin}}\to \sset_*$, from finite pointed simplicial sets to pointed simplicial sets, also play a proeminent role in Goodwillie calculus. The projective model structure on this category has been established in \cite{CD}, and is used to develop a model theoretic framework for Goodwillie calculus, where $n$-exicisive approximations are seen as fibrant replacements. The injective one is the key of the definition of homotopy nilpotent groups by Biedermann and Dwyer \cite{BD}. These model structures are recovered by Theorem~\ref{injproj}, as explained in Example \ref{sSet}. The difficulty to find an explicit reference for the injective case in subsequent study of homotopy nilpotency by Chorny, Costoya, Scherer, and Viruel in \cite{CS} and \cite{CSV} was actually the starting point of this project. 

However, Theorem \ref{injproj} does not apply in as many situations as we wish. Therefore, we give other conditions in Theorem \ref{under} for the existence of these injective and projective model structures, in the more special case where $\Nu$ admits a model structure and the model structure on $\ca$ is enriched over $\Nu$. This result permits to recover, among others, the injective and projective model structures on categories of dg modules over a dg category given by Keller in \cite[Theorem 3.2]{Keller}, and the injective model structure on categories of modules over a symmetric ring spectrum of \cite[Corollary 5.0.1]{HKRS}.

All the examples given so far are combinatorial. Nevertheless, the Hurewicz model structure on a category of chain complexes provides an example of an accessible model structure which is not cofibrantly generated (see \cite{HKRS} and \cite{CH}). As discussed in Example~\ref{Hur}, the injective and projective model structures on a category of dg diagrams in chain complexes endowed with the Hurewicz model structure exist. This illustrates why the accessibility condition is more general than the combinatorial one in our setting.

\subsection*{The non-enriched case}
Before presenting an outline of our arguments in the enriched case, let us explain how the proof of the existence of the injective and projective model structures in the non-enriched case works, using the methods of lifting model structures of \cite{HKRS} and \cite{GKR}. This will help us to highlight the similarities, but also the key differences between the two situations. Given a model structure $(\cc, \cf, \cw)$ on a category $\cm$ and two adjunctions 
\begin{tz}
\node (D) at (-1,0) {$\ck$};
\node (M) at (2,0) {$\cm$};
\node (E) at (5,0) {$\cn$,};
\draw[->] (D) to [bend left=30] node[above] {$V$} (M);
\draw[->] (M) to [bend left=30] node[below] {$R$} (D);
\draw[->] (M) to [bend left=30] node[above] {$L$} (E);
\draw[->] (E) to [bend left=30] node[below] {$U$} (M);
\node at (0.5,0) {$\bot$};
\node at (3.5,0) {$\bot$};
\end{tz}
one could lift the model structure on $\cm$ along the left or right adjoint to build model structures on $\ck$ and $\cn$ respectively. In other words, the class of cofibrations and weak equivalences in the left-lifted model structure on $\ck$ are given by $V^{-1}\cc$ and $V^{-1}\cw$ respectively. Dually, the class of fibrations and weak equivalences in the right-lifted model structure on $\cn$ are given by $U^{-1}\cf$ and $U^{-1}\cw$ respectively. However, these lifted model structures do not always exist. Garner, K\polhk{e}dziorek, and Riehl give a proof of the Acyclicity Theorem in \cite[Corollary 2.7]{GKR}, which states that these lifted model structures exist, when~$\ck$ and~$\cn$ are locally presentable categories, $\cm$ is an accessible model category, and the \emph{Acyclicity conditions} hold. This result was initially proved in \cite[Corollary 3.3.4]{HKRS}, but with a stronger notion of accessibility. The Acyclicity conditions correspond to the following conditions:
\begin{rome}
\item the morphisms in $\ck$ which have the right lifting property with respect to $V^{-1} \cc$ are contained in $V^{-1}\cw$, for the left-lifting case, and
\item the morphisms in $\cn$ which have the left lifting property with respect to $U^{-1} \cf$ are contained in $U^{-1} \cw$, for the right-lifting case.
\end{rome}

In \cite[Theorem 3.4.1]{HKRS}, this Acyclicity Theorem is used to prove the existence of the injective and projective model structures on a category of diagrams $\cm^D$, where $D$ is any small category and $\cm$ is an accessible model category. The injective and projective model structures can be seen as left- and right-lifted model structures from the Kan extension adjunctions
\begin{tz}
\node (EC) at (-1,0) {$\cm^D$};
\node (A) at (2,0) {$\cm^{\ObD}$,};
\path[->] (EC) edge node[b] {$i^*$} (A);
\path[->] (A) edge [bend right=45] node[above] {$i_!$} (EC);
\path[->] (A) edge [bend left=45] node[below] {$i_*$} (EC);
\node at (0.5,0.5) {$\bot$};
\node at (0.5, -0.5) {$\bot$};
\end{tz}
where $\ObD$ denotes the discrete category of objects of $D$, $i\colon \ObD \to D$ is the canonical inclusion functor, and $\cm^{\ObD}$ has the pointwise model structure coming from the one on~$\cm$. In particular, the diagram categories $\cm^D$ and $\cm^{\ObD}$ satisfy the conditions of the Acyclicity Theorem. To see this, note that $\cm^D$ is locally presentable since $\cm$ is so, and that the pointwise model structure on $\cm^{\ObD}$ coming from the accessible model structure on $\cm$ is also accessible. Moreover, the Acyclicity conditions are straightforward, since every injective trivial fibration is in particular a pointwise trivial fibration, and dually every projective trivial cofibration is in particular a pointwise trivial cofibration. 

\subsection*{Outline}
In this paper, we apply the same methods in order to prove the existence of the injective and projective model structures on categories of enriched diagrams and enriched natural transformations. In order to apply the Acyclicity Theorem, things need to be formulated in an enriched setting. Suppose $(\Nu, \otimes, I)$ is a closed symmetric monoidal category. In Section \ref{sec2}, following \cite{BQR} and \cite{KEL2}, we introduce the notions of a locally presentable base $\Nu$, and of enriched local presentability for a $\Nu$-category~$\ca$, which implies the local presentability (in the non-enriched sense) of the category of enriched diagrams $D\to \ca$, where $D$ is a small $\Nu$-category. In Section~\ref{sec3}, we recall the notion of an accessible model category, and state the Acyclicity Theorem from \cite{GKR}. In Section~\ref{sec4}, we construct an enriched category of objects for a small $\Nu$-category~$D$, which plays a similar role to the discrete category of objects above, but in the enriched setting. Then, we prove Theorem \ref{injproj}, using the fact that the injective and projective model structures on a category of enriched diagrams $D\to \ca$ can be seen as left- and right-induced model structures from the enriched Kan extension adjunctions induced by the inclusion $\Nu$-functor $\ObD\to D$. The assumptions saying that tensoring with the hom-objects of~$D$ preserves cofibrations or trivial cofibrations in $\ca$ respectively, imply the Acyclicity conditions. More precisely, they imply that injective trivial fibrations are in particular pointwise trivial fibrations, and that projective trivial cofibrations are in particular pointwise trivial cofibrations respectively. In Section~\ref{sec5}, we introduce the notion of an enriched model category, and prove that the enrichment of the model structure on~$\ca$ transfers to the injective and projective model structures on a category of enriched diagrams $D\to \ca$. In this case, the assumptions on the hom-objects of $D$ for the existence of these model structures are always true when all hom-objects of $D$ are cofibrant in $\Nu$. In Section \ref{sec6}, we also state other conditions under which the Acyclicity Theorem applies when the model structure on $\ca$ is enriched, and the proof follows from the Cylinder and Path Object arguments of \cite{HKRS}. In particular, these conditions hold when all objects of $\ca$ are cofibrant or fibrant respectively, and the unit $I$ is cofibrant or fibrant in $\Nu$ respectively. In Section \ref{sec7}, we apply our results to categories of modules over an operad in $\Nu$, using their characterization in terms of $\Nu$-functors as given by Arone and Turchin in \cite{AT}. This application was suggested by Kathryn Hess. In an upcoming paper on configuration spaces of products with Ben Knudsen \cite{HK}, they are using these model structures on modules over simplicial operads. Finally, in Section~\ref{sec8}, we show that, if $\ca$ is left- or right-proper, then so are the injective and projective model structures on a category of enriched diagrams $D\to \ca$, again under assumptions on the hom-objects of $D$. This time, these assumptions imply that injective fibrations are in particular pointwise fibrations, and that projective cofibrations are in particular pointwise cofibrations respectively.

\subsection*{Notations}

Throughout the whole article, the following notations are used. Let $(\Nu, \otimes, I)$ be a closed symmetric monoidal category. There is a $2$-functor $(-)_0\colon \Nu$-$\CAT\to \CAT$ sending a $\Nu$-category $\ca$ to its underlying category $\ca_0$ which has the same objects as $\ca$ (see \cite[Proposition 3.5.10]{RIE}). Let $\ca$ be a $\Nu$-category. For $A,B\in \ca$, we denote by 
\begin{itemize}
\item $\ca(A,B)$, the hom-object in $\Nu$ from $A$ to $B$, and
\item $\ca_0(A,B)=\{ I\to \Hom_\ca(A,B)\}$, the underlying set of morphisms from $A$ to $B$.
\end{itemize}
Note that, since $\Nu$ is closed, it is enriched over itself, and its underlying category is $\Nu$ (see \cite[Lemma 3.4.9]{RIE}). If $D$ is a small $\Nu$-category, we denote by 
\begin{itemize}
\item $[D, \ca]$, the $\Nu$-category of $\Nu$-functors from $D$ to $\ca$, and
\item $[D, \ca]_0$, the ordinary category of $\Nu$-functors from $D$ to $\ca$, and $\Nu$-natural transformations between them. 
\end{itemize}

\subsection*{Acknowledgements}  

I would like to specially thank my advisor, Jérôme Scherer, for recommending me this problem, and for his help throughout the writing of this paper. I~would also like to thank Kathryn Hess, Magdalena K\polhk edziorek, Fernando Muro, Emily Riehl, and Oliver Röndigs for our many useful discussions which led to improvement and generalization of the results in this text. In particular, Magdalena K\polhk{e}dziorek recommended to use the Cylinder Object argument, which gave rise to the new conditions and applications in Section \ref{sec6}, and Kathryn Hess suggested the application of Section~\ref{sec7}. Finally, I would like to thank Alexander Campbell, and Irakli Patchkoria for suggesting other examples of applications, Kay Werndli for giving the idea for the proof of Proposition \ref{local}, and Dimitri Zaganidis for pointing out the reference to Lurie's theorem.

\section{Enriched local presentability} \label{sec2}

Let $(\Nu, \otimes, I)$ be a closed symmetric monoidal category. Given a $\Nu$-category $\ca$, we want to find conditions on $\Nu$ and~$\ca$ under which the category $[D, \ca]_0$ is locally presentable, for every small $\Nu$-category $D$. Based on \cite{BQR}, we define locally presentable bases, which are locally presentable closed symmetric monoidal categories whose local presentability is compatible with the monoidal structure. Then, we extend the notion of local presentability to the enriched setting. In particular, if $\Nu$ is a locally presentable base, every $\Nu$-category $[D,\Nu]$, where $D$ is a small $\Nu$-category, is locally presentable in the enriched sense. Moreover, since every enriched locally presentable $\Nu$-category can be seen as a full reflective $\Nu$-subcategory of a $\Nu$-category of the form $[D,\Nu]$ by a result in \cite{KEL2}, if $\ca$ is locally presentable in the enriched sense, then so is every $\Nu$-category of enriched diagrams in $\ca$. The fact that the underlying category of an enriched locally presentable $\Nu$-category is locally presentable implies finally that the category $[D,\ca]_0$ is locally presentable, for every small $\Nu$-category $D$.

\begin{defn} \label{base}
Let $\alpha$ be a regular cardinal. A \textbf{locally \text{$\alpha$-presentable} base} is a cocomplete closed symmetric monoidal category $(\Nu, \otimes, I)$ which admits a strongly generating family of $\text{$\alpha$-presentable}$ objects containing the unit $I$ and closed under tensor products. We~say that $\Nu$ is a \textbf{locally presentable base} if it is a locally $\alpha$-presentable base for some regular cardinal $\alpha$. 
\end{defn}

Let $(\Nu, \otimes, I)$ be a locally presentable base. We extend the notion of local presentability to the enriched setting with the help of enriched hom-functors and enriched colimits.

\begin{defn}
Let $\ca$ be a $\Nu$-category, and $\alpha$ be a regular cardinal. An object $A$ of $\ca$ is \textbf{$\alpha$-$\Nu$-presentable} if the representable $\Nu$-functor $\ca(A,-)\colon \ca\to \Nu$ preserves $\alpha$-filtered $\Nu$-colimits. 
\end{defn}

\begin{defn} \label{loc}
Let $\alpha$ be a regular cardinal. A $\Nu$-category $\ca$ is \textbf{locally $\alpha$-$\Nu$-presentable} if it is $\Nu$-cocomplete and admits a strongly $\Nu$-generating family of $\alpha$-$\Nu$-presentable objects. We say that $\ca$ is \textbf{locally $\Nu$-presentable} if it is locally $\alpha$-$\Nu$-presentable for some regular cardinal $\alpha$. 
\end{defn}

Similarly to the non-enriched case, there is a characterization of locally $\Nu$-presentable $\Nu$-categories in terms of full reflective $\Nu$-subcategories of some $\Nu$-category of enriched diagrams in $\Nu$, where reflective here means that the inclusion $\Nu$-functor has a left adjoint, and together they form an enriched adjunction. 

\begin{prop} [\!{\cite[Corollary 7.3]{KEL2}}] \label{criterion}
Let $\alpha$ be a regular cardinal. A $\Nu$-category~$\ca$ is locally $\alpha$-$\Nu$-presentable if and only if it is a full reflective $\Nu$-subcategory of some $\text{$\Nu$-category}$ $[K, \Nu]$, where $K$ is a small $\Nu$-category and the inclusion $\ca\to [ K, \Nu]$ preserves $\alpha$-filtered $\Nu$-colimits. 
\end{prop}

\begin{rem}
In particular, the $\Nu$-category $[D,\Nu]$ is locally $\Nu$-presentable for every small $\Nu$-category~$D$. 
\end{rem}

In the case of ordinary categories, a category of diagrams in a locally presentable category is also locally presentable. Using the characterization above, we show that it is also true in the enriched setting. 

\begin{prop} \label{local}
Let $\ca$ be a locally $\alpha$-$\Nu$-presentable $\Nu$-category, for some regular cardinal~$\alpha$, and let $D$ be a small $\Nu$-category. Then the $\Nu$-category $[D, \ca]$ is also locally $\alpha$-$\Nu$-presentable. 
\end{prop}

\begin{proof}
Consider the $2$-functor $[D,-]\colon \Nu$-$\CAT\to \Nu$-$\CAT$ sending a $\Nu$-category $\cb$ to the $\Nu$-category of $\Nu$-functors $[D,\cb]$. In particular, this $2$-functor preserves enriched adjunctions. Hence, since $\ca$ is a full reflective $\Nu$-subcategory of some $[K, \Nu]$, where $K$ is a small $\text{$\Nu$-category}$, by Proposition \ref{criterion}, it follows from applying $[D,-]$ that $[D, \ca]$ is a full reflective $\Nu$-subcategory of $[D, [K, \Nu]]\cong [D\otimes K, \Nu]$. Note that $D\otimes K$ is also a small $\text{$\Nu$-category}$. Moreover, since $\Nu$-colimits are computed pointwise in $[D, \ca]$ and in $[D, [K, \Nu]]$ (see \cite[Section 3.3]{KEL}), the induced inclusion $\Nu$-functor $[D, \ca]\to [D, [K, \Nu]]$ preserves $\alpha$-filtered $\Nu$-colimits, as the inclusion $\Nu$-functor $\ca\to [K,\Nu]$ does so by Proposition~\ref{criterion}. This shows that $[D, \ca]$ is also locally $\alpha$-$\Nu$-presentable.  
\end{proof}  

Finally, in order to show that $[D,\ca]_0$ is locally presentable when the $\Nu$-category $\ca$ is locally $\text{$\Nu$-presentable}$, it remains to state a last result saying that the underlying category of a locally $\Nu$-presentable $\Nu$-category is locally presentable in the ordinary sense. 

\begin{prop} [\!{\cite[Proposition 6.6]{BQR}} ] \label{pres}
Let $\ca$ be a $\Nu$-cocomplete $\Nu$-category and denote by $\ca_0$ the underlying category of $\ca$. If $\ca$ is locally $\alpha$-$\Nu$-presentable for some regular cardinal~$\alpha$, then $\ca_0$ is locally $\alpha$-presentable in the ordinary sense.
\end{prop}

\begin{cor} \label{locpres}
Let $\ca$ be a locally $\Nu$-presentable $\Nu$-category, and let $D$ be a small $\text{$\Nu$-category}$. The category $[D, \ca]_0$ is locally presentable (in the ordinary sense).
\end{cor}

\begin{proof}
This follows directly from Propositions \ref{local} and \ref{pres}.
\end{proof}

\section{Induced model structures} \label{sec3}

Given an adjunction between locally presentable categories
\begin{tz}
\node (D) at (-1,0) {$\ck$};
\node (M) at (2,0) {$\cm$,};
\draw[->] (D) to [bend left=30] node[above] {$V$} (M);
\draw[->] (M) to [bend left=30] node[below] {$R$} (D);
\node at (0.5,0) {$\bot$};
\end{tz}
an accessible model structure $(\cc,\cf,\cw)$ on $\cm$ can be lifted along the left adjoint $U$ to give rise to a model structure on $\ck$. This lifted model structure has $U^{-1}\cc$ as the class of cofibrations and $U^{-1}\cw$ as the class of weak equivalences, while the fibrations are defined by their lifting property with respect to trivial cofibrations. In particular, the injective model structure on a category of diagrams in an accessible model category can be seen as such a lifted model structure, since its cofibrations and weak equivalences are defined pointwise, and its fibrations are defined by their lifting property. Of course, there is a dual version to this construction, where the model structure is lifted along the right adjoint. The projective model structure on a category of diagrams in an accessible model category is then an example of this dual construction. For the enriched case, we want to apply the same methods. Hence, we recall in this section the notions of accessible weak factorization systems and accessible model categories, and state the results from \cite{GKR} about the existence of liftings of such weak factorization systems and model structures.

\begin{notation}
Let $\cc$ denote a class of morphisms in a category $\cm$. The class $\cc^{\,\lift}$ is the class of morphisms in $\cm$ which have the right lifting property with respect to all morphisms in $\cc$, and the class $^{\lift\,}\cc$ is the class of morphisms in $\cm$ which have the left lifting property with respect to all morphisms in $\cc$.
\end{notation}

A \emph{weak factorization system} $(\mathcal{L}, \mathcal{R})$ on a category $\cm$ consists of two classes of morphisms $\mathcal L$ and $\mathcal R$ in $\cm$ such that 
\[ \begin{array}{lcr}
\mathcal L = \, ^{\lift}\mathcal R & \text{ and } & \mathcal R = \mathcal L^{\,\lift},
\end{array} \]
and every morphism in $\cm$ factors as a morphism in $\mathcal L$ followed by a morphism in $\mathcal R$. In particular, if $(\cc, \cf, \cw)$ is a model structure on $\cm$, then $(\cc\cap \cw, \cf)$ and $(\cc, \cf\cap \cw)$ form weak factorization systems on $\cm$. In particular, the notion of accessible weak factorization systems induces the notion of accessibility for a model structure. 

\begin{defn}
A weak factorization system $(\mathcal{L}, \mathcal{R})$ on a category $\cm$ is \textbf{accessible} if~$\cm$ is locally presentable and there is a functorial factorization 
\[ \begin{array}{lcr}
A\stackrel{f}{\longrightarrow} B & \mapsto & A\stackrel{Lf}{\longrightarrow} Ef \stackrel{Rf}{\longrightarrow} B
\end{array} \]
with $Lf\in \mathcal{L}$ and $Rf\in \mathcal{R}$ such that the functor $E\colon \cm^2\to \cm$ is accessible, i.e.~it preserves $\alpha$-filtered colimits for some regular cardinal $\alpha$. 
\end{defn}

\begin{defn}
A model category $(\cm, \cc, \cf, \cw)$ is \textbf{accessible} if the weak factorization systems $(\cc\cap \cw, \cf)$ and $(\cc, \cf\cap \cw)$ are accessible.
\end{defn}

\begin{rem}
Given the definition of an accessible weak factorization system, an accessible model category is in particular locally presentable. 
\end{rem}

By results in \cite{GKR}, a weak factorization system can be lifted along the left or right adjoint of an adjunction when it is accessible. 

\begin{defn} \label{lifts}
Let $(\mathcal L, \mathcal R)$ be a weak factorization system on $\cm$ and suppose we have the following adjunctions 
\begin{tz}
\node (D) at (-1,0) {$\ck$};
\node (M) at (2,0) {$\cm$};
\node (E) at (5,0) {$\cn$.};
\draw[->] (D) to [bend left=30] node[above] {$V$} (M);
\draw[->] (M) to [bend left=30] node[below] {$R$} (D);
\draw[->] (M) to [bend left=30] node[above] {$L$} (E);
\draw[->] (E) to [bend left=30] node[below] {$U$} (M);
\node at (0.5,0) {$\bot$};
\node at (3.5,0) {$\bot$};
\end{tz}
The \textbf{left-lifting} of $(\mathcal L, \mathcal R)$ along $V$, if it exists, is the weak factorization system on $\ck$ given by
\[ (\Ll, \Rl)=(V^{-1}\mathcal L, (V^{-1}\mathcal L)^{\lift}).\] 
The \textbf{right-lifting} of $(\mathcal L, \mathcal R)$ along $U$, if it exists, is the weak factorization system on $\cn$ given by
\[ (\Lr, \Rr)=(^{\lift} (U^{-1}\mathcal R), U^{-1}\mathcal R).\]
\end{defn}

\begin{theorem}[\!{\cite[Theorem 2.6]{GKR}}] \label{existence}
Let $(\mathcal L, \mathcal R)$ be an accessible weak factorization system on $\cm$, and let $\ck$ and $\cn$ be locally presentable categories. Suppose we have adjunctions as in Definition \ref{lifts}. Then $(\mathcal L, \mathcal R)$ admits a left-lifting along $V$ and a right-lifting along $U$. 
\end{theorem}

Suppose we are given adjunctions as in Definition \ref{lifts}, and suppose moreover that there is a model structure $(\cc, \cf,\cw)$ on $\cm$. The \emph{left-induced model structure} on $\ck$, if it exists, is given by $(V^{-1}\cc, (V^{-1}(\cc\cap\cw))^{\lift},V^{-1}\cw)$, and the \emph{right-induced model structure} on~$\cn$, if it exists, is given by $(^{\lift}(U^{-1}(\cf\cap\cw)), U^{-1}\cf,U^{-1}\cw)$. The existence of these induced model structures in the case of an accessible model category is given by the Acyclicity conditions. 

\begin{theorem} [\!{\cite[Corollary 3.3.4]{HKRS}}, {\cite[Corollary 2.7]{GKR}}] \label{ind}
Let $(\cm, \cc, \cf, \cw)$ be an accessible model category, and let $\ck$ and $\cn$ be two locally presentable categories. Suppose we have the following adjunctions
\begin{tz}
\node (D) at (-1,0) {$\ck$};
\node (M) at (2,0) {$\cm$};
\node (E) at (5,0) {$\cn$.};
\draw[->] (D) to [bend left=30] node[above] {$V$} (M);
\draw[->] (M) to [bend left=30] node[below] {$R$} (D);
\draw[->] (M) to [bend left=30] node[above] {$L$} (E);
\draw[->] (E) to [bend left=30] node[below] {$U$} (M);
\node at (0.5,0) {$\bot$};
\node at (3.5,0) {$\bot$};
\end{tz}
\begin{rome}
\item The left-induced model structure on $\ck$ exists if and only if 
\[ (V^{-1} \cc)^{\lift} \subseteq V^{-1}\cw. \]
\item The right-induced model structure on $\cn$ exists if and only if 
\[ ^{\lift} (U^{-1} \cf)\subseteq U^{-1} \cw. \]
\end{rome}
\end{theorem}

\section{Injective and projective model structures} \label{sec4}

In this section, we prove the main theorem (Theorem \ref{injproj}) on the existence of injective and projective model structures for enriched diagram categories. Let $(\Nu,\otimes,I)$ be a locally presentable base. Suppose $\ca$ is a locally $\Nu$-presentable $\Nu$-category with an accessible model structure on its underlying category $\ca_0$, and let $D$ be a small $\text{$\Nu$-category}$. The injective and projective model structures on $[D,\ca]_0$ can be seen as left- and right-induced model structures from the one on a category of diagrams $[\ObD,\ca]_0$, which admits a pointwise model structure coming from the one of $\ca_0$. The first step is to define this $\Nu$-category $\ObD$ of objects of $D$ in such a way that it plays a similar role to the discrete category of objects in the non-enriched case. Then it remains to find conditions which imply the Acyclicity conditions of Theorem \ref{ind}. In the non-enriched case, these conditions are straightforward since injective trivial fibrations are in particular pointwise trivial fibrations, and dually projective trivial cofibrations are in particular pointwise trivial cofibrations. This follows among others from the fact that the coproduct over a set of a (trivial) cofibration is also a (trivial) cofibration. But the coproduct by a set corresponds to a tensor over the category of sets, for every ordinary category. Generalizing this to the enriched setting, the Acyclicity conditions for $[D,\ca]_0$ are satisfied whenever tensoring with the hom-objects of $D$ preserves cofibrations or trivial cofibrations respectively. Under these mild conditions, we can prove that the injective and projective model structures on $[D,\ca]_0$ exist. 

\begin{notation}
We denote by $\varnothing$ the initial object of $\Nu$, which exists since $\Nu$ is cocomplete. 
\end{notation}

\begin{defn} \label{enrichedob}
Let $D$ be a small $\Nu$-category. The \textbf{enriched category of objects} of~$D$ is the $\Nu$-category $\ObD$ with the same objects as $D$, in which the hom-objects are given by
\[ \begin{array}{lcr}
\ObD(d,d') = \varnothing, & \text{ and } & \ObD(d,d) = I,
\end{array} \]
for every $d\neq d'\in D$, and the identity morphisms are given by
\[ \id_d=\id_I\colon I\longrightarrow \ObD(d,d)=I, \]
for every $d\in D$. There is an inclusion $i\colon \ObD \to D$ given by the $\Nu$-functor which is the identity on objects, and where
\[ i_{d,d'}\colon \ObD(d,d')=\varnothing \longrightarrow D(d,d') \]
corresponds to the unique morphism in $\Nu$ from the initial object to $D(d,d')$, for every~${d\neq d'\in D}$, and
\[ i_{d,d}\colon \ObD(d,d)=I\longrightarrow D(d,d) \]
corresponds to the identity morphism $\id_d$ in $D$, for every $d\in D$.  
\end{defn}

The following lemma motivates the definition of the enriched category of objects by saying that the $\Nu$-category $\ObD$ in the enriched case plays a role similar to the one of the discrete category of objects in the non-enriched case. 

\begin{lemme}
Let $\ca$ be a $\Nu$-category, and let $D$ be a small $\Nu$-category. Then 
\begin{rome}
\item a $\Nu$-functor $F\colon \ObD\to \ca$ corresponds to a family of objects $\{ Fd\}_{d\in D}$ in $\ca$, and
\item a $\Nu$-natural transformation $\alpha\colon F\Rightarrow G$ in $[\ob D, \ca]_0$ corresponds to a family of morphisms $\{\alpha_d\colon Fd\to Gd\}_{d\in D}$ in $\ca_0$, without further conditions. 
\end{rome}
\end{lemme}

\begin{proof}
Let $F\colon \ObD \to \ca$ be a $\Nu$-functor. For every $d\neq d'\in D$, the morphism 
\[ F_{d,d'}\colon \ObD(d,d')=\varnothing \longrightarrow \ca(Fd,Fd') \]
corresponds to the unique morphism in $\Nu$ from the initial object to $\ca(Fd,Fd')$ and, for every $d\in D$, the morphism 
\[ F_{d,d}\colon \ObD(d,d)=I \longrightarrow \ca(Fd,Fd) \]
is the identity morphism $\id_{Fd}$ in $\ca$, since the following diagram commutes
\begin{tz}
\node (I) at (-1,0) {$I$};
\node (D) at (2,1) {$\ObD(d,d)=I$};
\node (V) at (2,-1) {$\ca(Fd,Fd)$.};

\draw[->] (I) to node[above] {$\id_d=\id_I\,\,\,\,\,\,\,\,\,\,\,\,$} (D);
\draw[->] (D) to node[right] {$F_{d,d}$} (V);
\draw[->] (I) to node[below] {$\id_{Fd}$} (V);
\end{tz}
Since these morphisms between the hom-objects are uniquely determined, the $\Nu$-functor~$F$ corresponds to the family of objects $\{ Fd\}_{d\in D}$ in $\ca$.

Let $\alpha\colon F\Rightarrow G$ be a $\Nu$-natural transformation in $[\ob D, \ca]_0$. Recall that a morphism $I\to \ca(Fd,Gd)$ in $\Nu$ corresponds to a morphism $Fd\to Gd$ in $\ca_0$, for every $d\in D$, by definition of the underlying category of $\ca$. Moreover, the following diagrams trivially commute, for every $d\neq d'\in D$. 
\begin{tz}
\node (A) at (0,2) {$\ObD(d,d')=\varnothing$};
\node (B) at (4,2) {$\ca(Fd,Fd')$};
\node (C) at (0,0) {$\ca(Gd,Gd')$};
\node (D) at (4,0) {$\ca(Fd,Gd')$};

\draw[->] (A) to node[above] {$F_{d,d'}$} (B);
\draw[->] (A) to node[left] {$G_{d,d'}$} (C);
\draw[->] (B) to node[right] {$(\alpha_{d'})_*$} (D);
\draw[->] (C) to node[below] {$(\alpha_d)^*$} (D);

\node (A) at (7.5,2) {$\ObD(d,d)=I$};
\node (B) at (12,2) {$\ca(Fd,Fd)$};
\node (C) at (7.5,0) {$\ca(Gd,Gd)$};
\node (D) at (12,0) {$\ca(Fd,Gd)$};

\draw[->] (A) to node[above] {$F_{d,d}=\id_{Fd}$} (B);
\draw[->] (A) to node[left] {$G_{d,d}=\id_{Gd}$} (C);
\draw[->] (B) to node[right] {$(\alpha_{d})_*$} (D);
\draw[->] (C) to node[below] {$(\alpha_d)^*$} (D);
\draw[->] (A) to node[above] {$\alpha_d$} (D);
\end{tz}
This shows that $\alpha\colon F\Rightarrow G$ corresponds to a family of morphisms $\{\alpha_d\colon Fd\to Gd\}_{d\in D}$ in $\ca_0$, without further conditions.
\end{proof}

Hence, if the underlying category $\ca_0$ of $\ca$ admits a model structure, there is a model structure on $[\ob D, \ca]_0$ coming from the one of~$\ca_0$, in which cofibrations, fibrations and weak equivalences are defined pointwise.  

Moreover, if $\ca$ is $\Nu$-complete and $\Nu$-cocomplete, we have the following adjunctions
\begin{tz}
\node (EC) at (-1,0) {$[D, \ca]_0$};
\node (ED) at (2,0) {$[\ob D, \ca]_0$,};
\path[->] (EC) edge node[b] {$i^*$} (ED);
\path[->] (ED) edge [bend right=45] node[above] {$i_!$} (EC);
\path[->] (ED) edge [bend left=45] node[below] {$i_*$} (EC);
\node at (0.5,0.5) {$\bot$};
\node at (0.5, -0.5) {$\bot$};
\end{tz}
where $i^*\colon [D,\ca]_0 \rightarrow [\ObD, \ca]_0$ is the precomposition functor by $i\colon \ObD\to D$, and 
\[ i_!, i_*\colon [\ObD,\ca]_0 \longrightarrow [D,\ca]_0 \]
are the underlying functors of the \emph{enriched} left and right Kan extension functors. Our aim is to lift the injective and projective model structures on $[D, \ca]_0$ from the pointwise model structure on~$[\ob D, \ca]_0$ through these adjunctions. 

The final step is to find conditions that imply the Acyclicity conditions. These conditions are about tensoring with the hom-objects of $D$. Note that, if $\ca$ is $\Nu$-complete and $\Nu$-cocomplete, it is tensored and cotensored over $\Nu$, since these are particular enriched colimits and limits. 

\begin{theorem} \label{injproj}
Let $(\Nu, \otimes, I)$ be a locally presentable base. Suppose $\ca$ is a $\Nu$-complete locally $\Nu$-presentable $\Nu$-category such that its underlying category $\ca_0$ admits an accessible model structure, and let $D$ be a small $\Nu$-category. 
\begin{rome}
\item If the functors $-\otimes D(d,d')\colon \ca_0\to \ca_0$ preserve cofibrations for all $d,d'\in D$, the injective model structure on the category $[D, \ca]_0$ exists. 
\item If the functors $-\otimes D(d,d')\colon \ca_0\to \ca_0$ preserve trivial cofibrations for all ${d,d'\in D}$, the projective model structure on the category $[D, \ca]_0$ exists. 
\end{rome}
\end{theorem}

\begin{proof}
We prove (i). The proof for (ii) is dual. Let $\ObD$ and $i\colon \ObD\to D$ be as in Definition \ref{enrichedob}. Let $(\cc,\cf,\cw)$ denote the accessible model structure on $\ca_0$. The weak factorization system $(\cc,\cf\cap \cw)$ induces a pointwise weak factorization system $ (\cc^{\ObD}, (\cf\cap\cw)^{\ObD})$ on $[\ob D,\ca]_0$. By Corollary \ref{locpres}, the category $[D, \ca]_0$ is locally presentable. Moreover, note that the category~$[\ob D, \ca]_0$ is also an accessible model category with the pointwise model structure coming from the accessible model structure of $\ca_0$. Since $\ca$ is $\Nu$-complete, we have an adjunction
\begin{tz}
\node (D) at (-1,0) {$[D, \ca]_0$};
\node (M) at (2,0) {$[\ob D, \ca]_0$,};
\draw[->] (D) to [bend left=30] node[above] {$i^*$} (M);
\draw[->] (M) to [bend left=30] node[below] {$i_*$} (D);
\node at (0.5,0) {$\bot$};
\end{tz}
where $i^*$ is the precomposition functor, and $i_*$ is the underlying functor of the enriched right Kan extension functor. So we can apply Lemma \ref{existence}, and obtain a left-lifting $(\cc^{\inj}, (\cf\cap\cw)^{\inj})$ on $[D,\ca]_0$ from the weak factorization system $(\cc^{\ObD}, (\cf\cap\cw)^{\ObD})$ on $[\ObD, \ca]_0$. By Theorem~\ref{ind}, it remains to show the Acyclicity condition 
\[((i^*)^{-1}\cc)^{\lift}\subseteq (i^*)^{-1}\cw.\]
Let $\eta\colon F\Rightarrow G\in ((i^*)^{-1}\cc)^{\lift}$. We show that $\eta$ is in particular a pointwise trivial fibration. To see this, for each $d\in D$, we show that $\eta_d\colon Fd\to Gd$ has the right lifting property with respect to every cofibration in $\ca_0$. Let $i\colon A\to B$ be a cofibration in~$\ca_0$, and fix~$d\in D$. By the enriched Yoneda lemma and since $\ca$ is tensored over $\Nu$,
\[ \begin{array}{lcl}
\ca(A, Fd)&\cong & [D,\Nu](D(d,-), \ca(A,F-))\\
&\cong & [D,\ca](A\otimes D(d,-), F).
\end{array} \]
Hence $\eta_d$ has the right lifting property with respect to $i$ in $\ca_0$ if and only if $\eta$ has the right lifting property with respect to $i\otimes D(d,-)$ in $[D,\ca]_0$. By assumption, each component of $i\otimes D(d,-)$ is a cofibration, since the functors $-\otimes D(d,d')\colon \ca_0\to \ca_0$ preserve cofibrations for all $d'\in D$. Since $\eta\in ((i^*)^{-1}\cc)^{\lift}$, it has the right lifting property with respect to~$i\otimes D(d,-)$, and thus $\eta$ is a pointwise trivial fibration. In particular,~$\eta$~is a pointwise weak equivalence, i.e.~$\eta\in (i^*)^{-1}\cw$.
\end{proof}

\begin{rem}
Suppose $\Nu$ admits a model structure. Since $\Nu$ is locally presentable in the enriched sense and it is $\Nu$-complete (see \cite[Section 3.2]{KEL}), if the functors ${-\otimes D(d,d')\colon \Nu\to \Nu}$ preserve cofibrations (resp.~trivial cofibrations) for all $d,d'\in D$, then the category $[D, \Nu]_0$ admits an injective (resp.~projective) model structure. 
\end{rem}

\begin{rem}
In Theorem \ref{injproj}, we need to assume that the $\Nu$-category $\ca$ is $\text{$\Nu$-complete}$ for the existence of the enriched right Kan extension $i_*\colon [\ob D, \ca]\to [D, \ca]$, while the $\text{$\Nu$-cocompleteness}$ (and hence the existence of the enriched left Kan extension) comes from the local $\Nu$-presentability of $\ca$.
\end{rem}

\begin{rem}
Since the functors $-\otimes X\colon \ca_0\to \ca_0$ and $(-)^X\colon \ca_0\to \ca_0$ form an adjunction for all $X\in \Nu$, the conditions in Theorem \ref{injproj} on the functors $-\otimes D(d,d')\colon \ca_0\to \ca_0$, for $d,d'\in D$, can equivalently be formulated as
\begin{rome}
\item the functors $(-)^{D(d,d')}\colon \ca_0\to\ca_0$ preserve trivial fibrations for all $d,d'\in D$, in the injective case, and
\item the functors $(-)^{D(d,d')}\colon \ca_0\to\ca_0$ preserve fibrations for all $d,d'\in D$, in the projective case.
\end{rome}
\end{rem}

\section{Enrichment of the model structure} \label{sec5}

Let $(\Nu,\otimes,I)$ be a locally presentable base and a model category. Given a $\Nu$-category~$\ca$, a model structure on its underlying category $\ca_0$ may be enriched over the model structure on $\Nu$. This gives rise to the notion of a $\Nu$-enriched model category, which is a generalization of the notion of a simplicial model category (see \cite[Definition 9.1.6]{Hir}). In this section, we prove that, if $\ca$ is a $\Nu$-enriched model category and $D$ is a small $\Nu$-category, then the injective and projective model structures on $[D,\ca]_0$ are again $\text{$\Nu$-enriched}$, when they exist.

\begin{defn} \label{emc}
A $\Nu$-category $\ca$ is a \textbf{$\Nu$-enriched model category} if 
\begin{description}
\item[{[MC1-5]}] its underlying category $\ca_0$ admits a model structure,
\item[{[MC6]}] it is tensored and cotensored over $\Nu$, and 
\item[{[MC7]}] if $i\colon A\to B$ is a cofibration in $\ca_0$ and $p\colon X\to Y$ is a fibration in $\ca_0$, the pullback corner map \[ (i^*, p_*)\colon \ca(B,X)\to \ca(A,X)\times_{\ca(A,Y)} \ca(B,Y) \] is a fibration in $\Nu$, which is trivial if either $i$ or $p$ is a weak equivalence. 
\end{description}
\end{defn}

We first check that the $\Nu$-category $[D,\ca]$ is tensored and cotensored over $\Nu$ whenever~$\ca$~is, for every small $\Nu$-category $D$. This shows that, when $\ca$ is a $\Nu$-enriched model category, then the $\Nu$-category $[D,\ca]$ satisfies in particular Axiom [MC6]. 

\begin{lemme} \label{cotens}
Let $\ca$ be a tensored and cotensored $\Nu$-category, and let $D$ be a small $\Nu$-category. The $\Nu$-category $[D,\ca]$ is tensored and cotensored over $\Nu$, with tensor and cotensor products defined pointwise. 
\end{lemme}

\begin{proof}
Let $F\colon D\to \ca$ be a $\Nu$-functor, and let $K\in \Nu$. The aim is to define a tensor product $F(-)\otimes K\colon D\to \ca$ and a cotensor product $F(-)^K\colon D\to \ca$. Consider the functor $-\otimes K\colon \ca\to \ca$ sending an object $A\in \ca$ to the tensor product $A\otimes K$ and such that, for $A,B\in \ca$, the morphism
\[ (-\otimes K)_{A,B}\colon \ca(A,B)\longrightarrow \ca(A\otimes K, B\otimes K) \]
is the adjunct of the evalutation morphism $\operatorname{ev}\otimes \id_K\colon \ca(A,B)\otimes A\otimes K\to B\otimes K$. Define the tensor product $F(-)\otimes K\colon D\to \ca$ to be the composite 
\[ D\stackrel{F}{\longrightarrow} \ca \xrightarrow{-\otimes K} \ca.\]
Then, consider the functor $(-)^K\colon \ca \to \ca$ sending an object $A\in \ca$ to the cotensor product $A^K$, and such that, for $A,B\in \ca$, the morphism
\[ ((-)^K)_{A,B}\colon \ca(A,B)\longrightarrow \ca(A^K, B^K) \]
is the adjunct to the morphism
\[ \ca(A,B)\otimes A^K\otimes K \xrightarrow{\id\otimes \eta} \ca(A,B)\otimes A \stackrel{\operatorname{ev}}{\longrightarrow} B, \]
where $\eta\colon A^K\otimes K\to A$ is the adjunct to the identity morphism of $A^K$. 
Define the cotensor product $F(-)^K\colon D\to \ca$ to be the composite
\[ D\stackrel{F}{\longrightarrow} \ca \stackrel{(-)^K}{\longrightarrow} \ca.\]
These constructions satisfy the expected adjunction properties. The proof is left to the reader.  
\end{proof}

To prove that the injective and projective model structures on $[D,\ca]$, if they exist, are $\Nu$-enriched whenever the one on $\ca$ is, it remains to show that Axiom [MC7] is satisfied. Since cofibrations and weak equivalences are defined pointwise in the injective model structure on $[D,\ca]$, and fibrations and weak equivalences are defined pointwise in the projective one, it is convenient to have equivalent statements to axiom [MC7], where only (trivial) cofibrations or only (trivial) fibrations of the enriched model structure appear. This allows us to check pointwise that the category $[D,\ca]$ satisfies this last axiom. Next lemma is an adaptation of \cite[Proposition 9.3.7]{Hir} to the general setting. 

\begin{lemme} \label{Cond}
Let $\ca$ be a tensored and cotensored $\Nu$-category, whose underlying category~$\ca_0$ admits a model structure. Axiom [MC7] of Definition \ref{emc} is equivalent to each of the following statements. 
\begin{rome}
\item If $i\colon A\to B$ is a cofibration in $\ca_0$ and $j\colon K\to L$ is a cofibration in $\Nu$, the pushout corner map
\[ i\star j\colon A\otimes L\amalg_{A\otimes K} B\otimes K\to B\otimes L \]
is a cofibration in $\ca_0$, which is trivial if either $i$ or $j$ is a weak equivalence. 
\item If $p\colon X\to Y$ is a fibration in $\ca_0$ and $j\colon K\to L$ is a cofibration in $\Nu$, the pullback corner map
\[ (j^*, p_*)\colon X^L\to X^K\times_{Y^L} Y^K \]
is a fibration in $\ca_0$, which is trivial if either $j$ or $p$ is a weak equivalence. 
\end{rome}
\end{lemme}

\begin{proof}
This follows immediately from the adjunctions
\[ \ca_0(K\otimes A, X)\cong \Nu(K, \ca(A,X)) \cong \ca_0(A, X^K), \]
where $A,X\in \ca$ and $K\in \Nu$. 
\end{proof}

If a $\Nu$-category $\ca$ admits all small conical limits and is cotensored over $\Nu$, then~$\ca$ is $\Nu$-complete (see \cite[Theorem 3.73]{KEL}). Hence a $\Nu$-enriched model category is $\text{$\Nu$-complete}$, and the $\Nu$-completeness assumption of Theorem \ref{injproj} can be removed. We call $\ca$ an \emph{accessible $\Nu$-enriched model category} if $\ca$ is a $\Nu$-enriched model category and the model structure on $\ca_0$ is accessible. 

\begin{theorem} \label{enriched}
Let $(\Nu, \otimes, I)$ be a locally presentable base, and a model category. Suppose~$\ca$ is a locally $\Nu$-presentable $\Nu$-category, which admits an accessible $\Nu$-enriched model structure, and let $D$ be a small $\Nu$-category. 
\begin{rome}
\item If the functors $-\otimes D(d,d')\colon \ca_0\to \ca_0$ preserve cofibrations for all $d,d'\in D$, the injective model structure on $[D,\ca]_0$ exists, and is again $\Nu$-enriched. 
\item If the functors $-\otimes D(d,d')\colon \ca_0\to \ca_0$ preserve trivial cofibrations for all ${d,d'\in D}$, the projective model structure on $[D,\ca]_0$ exists, and is again $\Nu$-enriched. 
\end{rome}
\end{theorem}

\begin{proof}
We show that the injective model structure satisfies condition (i) of Lemma~\ref{Cond}.  Let $\alpha\colon F\Rightarrow G$ be a cofibration in $[D,\ca]_0$, and $j\colon K\to L$ be a cofibration in~$\Nu$. Since the model structure is injective, the morphism $\alpha_d\colon Fd\to Gd$ is a cofibration in $\ca_0$, for every $d\in D$. The morphism $\alpha\star j\colon F\otimes L\amalg_{F\otimes K} G\otimes K\Rightarrow G\otimes L$ has components $(\alpha\star j)_d=(\alpha_d)\star j$, for $d\in D$, since tensor products and colimits are computed pointwise in $[D,\ca]_0$. Every component $(\alpha_d)\star j$, for $d\in D$, is a cofibration in $\ca_0$, since $\ca$ is a $\Nu$-enriched model category. Hence $\alpha\star j$ is a cofibration in $[D,\ca]_0$ with the injective model structure. Moreover, if either $\alpha$ or $j$ is a weak equivalence, the morphism $(\alpha_d)\star j$ is trivial, for every $d\in D$, and thus $\alpha\star j$ is also trivial. 

Similarly, one can show that the projective model structure satisfies condition (ii) of Lemma~\ref{Cond}. 
\end{proof}

\begin{rem} \label{bla}
In particular, if all hom-objects $D(d,d')$, for $d,d'\in D$, are cofibrant in $\Nu$, the functors $-\otimes D(d,d')\colon \ca_0\to \ca_0$ preserve all cofibrations and trivial cofibrations. To see this, apply Lemma \ref{Cond} (i) to every (trivial) cofibration $i\colon A\to B$ in $\ca_0$ and to the cofibration $\varnothing\to D(d,d')$ in $\Nu$, for $d,d'\in D$. In particular, if all objects in $\Nu$ are cofibrant, this condition is always satisfied. 
\end{rem}

This remark gives rise to our first example. Recall that every combinatorial model category is in particular accessible (see \cite[Corollary 3.1.7]{HKRS}). 

\begin{ex}[Simplicial enrichment] \label{sSet}
Take $(\Nu,\otimes,I)$ to be the closed symmetric mo\-no\-i\-dal category $(\sset,\times,\Delta^0)$ or $(\sset_*,\wedge, S^0)$ of (pointed) simplicial sets. These are locally presentable bases. The Quillen model structure, introduced in \cite{Qui}, endows~$\sset$ and~$\sset_*$ with a simplicial, combinatorial model structure, in which every object is cofibrant. Suppose $\ca$ is an enriched locally presentable and accessible simplicial model category. By Theorem \ref{enriched} and Remark \ref{bla}, for every small simplicial category~$D$, the injective and projective model structures on the category $[D,\ca]$ of simplicial diagrams exist and are simplicial. In particular, this applies to the folklore case where~${\ca=\sset,\sset_*}$ with the Quillen model structure. This also applies to ${\ca=\mathrm{Sp}^{\Sigma}}$ the category of symmetric spectra of simplicial sets enriched over simplicial sets, with one of the simplicial model structures, e.g.~the stable model structure or the stable/level projective and injective model structures (see \cite[III.3-4]{Sch}).
\end{ex}

\section{The Cylinder and Path Object arguments} \label{sec6}

Let $(\Nu,\otimes,I)$ be a locally presentable base, and a model category. If the unit $I$ is cofibrant in $\Nu$, one can find other conditions on a $\Nu$-enriched model category $\ca$ under which the injective model structure on $[D,\ca]_0$ exists, for every small $\Nu$-category $D$, e.g.~if all objects of $\ca$ are cofibrant, in contrast to Remark \ref{bla} where we assume that the hom-objects of $D$ are cofibrant in $\Nu$. There is a dual statement for the projective case. This gives more examples of categories of enriched diagrams which admit an injective or projective model structure. The proof follows from the Cylinder and Path Object arguments, which we first state here.  These statements are inspired by the original Quillen Path Object Argument of \cite{Qui}.

\begin{theorem}[\!{\cite[Theorem 2.2.1]{HKRS}}] \label{cyl}
Let $(\cm, \cc,\cf,\cw)$ be an accessible model category, and let $\ck$ be a locally presentable category. Suppose we have the following adjunction 
\begin{tz}
\node (D) at (-1,0) {$\ck$};
\node (M) at (2,0) {$\cm$.};
\draw[->] (D) to [bend left=30] node[above] {$V$} (M);
\draw[->] (M) to [bend left=30] node[below] {$R$} (D);
\node at (0.5,0) {$\bot$};
\end{tz}
If the following hold:
\begin{rome}
\item for every $X\in \ck$, there is a morphism $\phi_X\colon QX\to X$ in $\ck$ such that $V(\phi_X)$ is a weak equivalence and $V(QX)$ is cofibrant in $\cm$,
\item for every morphism $\alpha\colon X\to Y$ in $\ck$, there exists a morphism $Q\alpha\colon QX\to QY$ in $\ck$ such that $\alpha\circ \phi_X=\phi_Y\circ Q\alpha$, and 
\item for every $X\in \ck$, there is a factorization
\[ QX\sqcup QX\stackrel{i}{\longrightarrow} \Cyl(QX)\stackrel{w}{\longrightarrow} QX \]
of the fold map in $\ck$ such that $V(i)$ is a cofibration and $V(w)$ is a weak equivalence in $\cm$,
\end{rome}
then the Acyclicity condition $(V^{-1}\cc)^{\lift}\subseteq V^{-1}\cw$ holds. In particular, the left-induced model structure on $\ck$ exists.
\end{theorem}

\begin{theorem} \label{path}
Let $(\cm, \cc,\cf,\cw)$ be an accessible model category, and let $\cn$ be a locally presentable category. Suppose we have the following adjunction
\begin{tz}
\node (M) at (2,0) {$\cm$};
\node (E) at (5,0) {$\cn$.};
\draw[->] (M) to [bend left=30] node[above] {$L$} (E);
\draw[->] (E) to [bend left=30] node[below] {$U$} (M);
\node at (3.5,0) {$\bot$};
\end{tz}
If the following hold:
\begin{rome}
\item for every $X\in \cn$, there is a morphism $\psi_X\colon X\to RX$ in $\cn$ such that $U(\psi_X)$ is a weak equivalence and $U(RX)$ is fibrant in $\cm$,
\item for every morphism $\alpha\colon X\to Y$ in $\cn$, there exists a morphism $R\alpha\colon RX\to RY$ in $\cn$ such that $\psi_Y\circ \alpha=R\alpha \circ \psi_X$, and
\item for every $X\in \cn$, there exists a factorization
\[ RX \stackrel{w}{\longrightarrow} \Path(RX) \stackrel{p}{\longrightarrow} RX\times RX \]
of the diagonal map in $\cn$ such that $U(p)$ is a fibration and $U(w)$ is a weak equivalence in $\cm$. 
\end{rome}
then the Acyclicity condition $^{\lift} (U^{-1}\cf)\subseteq U^{-1}\cw$ holds. In particular, the right-induced model structure on $\cn$ exists.
\end{theorem}

Given a $\Nu$-enriched model category $\ca$ and a small $\Nu$-category $D$, we want to apply Theorems \ref{cyl} and \ref{path} to the underlying enriched Kan extension adjunctions 
\begin{tz}
\node (EC) at (-1,0) {$[D, \ca]_0$};
\node (ED) at (2,0) {$[\ob D, \ca]_0.$};
\path[->] (EC) edge node[b] {$i^*$} (ED);
\path[->] (ED) edge [bend right=45] node[above] {$i_!$} (EC);
\path[->] (ED) edge [bend left=45] node[below] {$i_*$} (EC);
\node at (0.5,0.5) {$\bot$};
\node at (0.5, -0.5) {$\bot$};
\end{tz} 
Conditions (iii) follow from the fact that the unit $I$ is cofibrant, respectively fibrant in~$\Nu$, as shown in Theorem \ref{under}. Hence, we introduce the notions of underlying cofibrant and fibrant replacements for the category $[D,\ca]_0$, which correspond to conditions (i) and (ii) of Theorems \ref{cyl} and \ref{path} respectively, applied to this setting. 

\begin{defn}
Let $\ca$ be a $\Nu$-enriched model category, and $D$ be a small $\Nu$-category. The category $[D,\ca]_0$ admits \textbf{underlying cofibrant replacements} if 
\begin{rome}
\item for every $\Nu$-functor $F\colon D\to \ca$, there is a $\Nu$-natural transformation ${\phi_F\colon QF\Rightarrow F}$ in $[D,\ca]_0$ such that, for every $d\in D$, $\phi_{Fd}\colon QFd\to Fd$ is a weak equivalence and $QFd$ is cofibrant in $\ca_0$, and
\item for every $\Nu$-natural transformation $\alpha\colon F\Rightarrow G$ in $[D,\ca]_0$, there exists a $\Nu$-natural transformation $Q\alpha\colon QF\Rightarrow QG$ in $[D,\ca]_0$ such that $\alpha\circ \phi_F=\phi_G\circ Q\alpha$. 
\end{rome}
The category $[D,\ca]_0$ admits \textbf{underlying fibrant replacements} if 
\begin{rome}
\item for every $\Nu$-functor $F\colon D\to \ca$, there is a $\Nu$-natural transformation ${\psi_F\colon F\Rightarrow RF}$ in $[D,\ca]_0$ such that, for every $d\in D$, $\psi_{Fd}\colon Fd\to RFd$ is a weak equivalence and $RFd$ is fibrant in $\ca_0$,
\item for every $\Nu$-natural transformation $\alpha\colon F\Rightarrow G$ in $[D,\ca]_0$, there exists a $\Nu$-natural transformation $R\alpha\colon RF\Rightarrow RG$ in $[D,\ca]_0$ such that $\psi_G\circ\alpha=R\alpha\circ \psi_F$. 
\end{rome}
\end{defn}

\begin{rem} \label{cofib}
In particular, if all objects of $\ca$ are cofibrant, then $[D,\ca]_0$ admits underlying cofibrant replacements, for every small $\Nu$-category $D$. To see this, given a $\text{$\Nu$-functor}$ $F\colon D\to \ca$, take $QF=F$ and $\phi_F=\id_F$. Dually, if all objects of $\ca$ are fibrant, then~$[D,\ca]_0$ admits underlying fibrant replacements, for every small $\Nu$-category $D$.
\end{rem}

If we suppose that the category of enriched diagrams $[D,\ca]_0$ has underlying cofibrant or fibrant replacements, by the Cylinder and Path Object arguments, it remains to show that the last condition holds when the unit $I$ is cofibrant or fibrant in $\Nu$ respectively. 

\begin{theorem} \label{under}
Let $(\Nu, \otimes, I)$ be a locally presentable base, and a model category. Suppose~$\ca$ is a locally $\Nu$-presentable $\Nu$-category, which admits an accessible $\Nu$-enriched model structure, and let $D$ be a small $\Nu$-category. 
\begin{rome}
\item If the unit $I\in \Nu$ is cofibrant, and the category $[D,\ca]_0$ admits underlying cofibrant replacements, then the injective model structure on $[D,\ca]_0$ exists, and is again $\Nu$-enriched. 
\item If the unit $I\in \Nu$ is fibrant, and the category $[D,\ca]_0$ admits underlying fibrant replacements, then the projective model structure on $[D,\ca]_0$ exists, and is again $\Nu$-enriched. 
\end{rome}
\end{theorem}

\begin{proof}
We prove (i). The proof for (ii) is dual. By Theorem \ref{cyl}, it is enough to show that, for every $\Nu$-functor $F\colon D\to \ca$ such that $Fd$ is cofibrant in $\ca$ for every $d\in D$, there exists a factorization
\[ F\sqcup F\Longrightarrow \Cyl(F)\Longrightarrow F \]
of the fold map in $[D,\ca]_0$ into a pointwise cofibration followed by a pointwise weak equivalence. Let $F\colon D\to \ca$ be a $\Nu$-functor such that $Fd$ is cofibrant in $\ca$ for every~${d\in D}$. Choose a good cylinder object 
\[ I\sqcup I \stackrel{i}{\longrightarrow} \Cyl(I) \stackrel{w}{\longrightarrow} I\]
for the unit $I\in \Nu$, where $i$ is a cofibration and $w$ is a weak equivalence in $\Nu$. Now apply~$F\otimes -$ to the sequence above, and get
\[ F\sqcup F\cong F\otimes (I\sqcup I) \xRightarrow{F\otimes i} F\otimes \Cyl(I) \xRightarrow{F\otimes w} F\otimes I\cong F. \]
For every $d\in D$, since $Fd$ is cofibrant in $\ca$, by Axiom [MC7], $Fd\otimes -\colon \Nu\to \ca$ preserves cofibrations and trivial cofibrations. It follows that $F\otimes i$ is a pointwise cofibration, since the tensor product is defined pointwise in $[D,\ca]_0$. Moreover, by Ken Brown's Lemma, $Fd\otimes -\colon \Nu\to\ca$ preserves weak equivalences between cofibrant objects, for every $d\in D$. Since $I$ is cofibrant in $\Nu$, the coproduct $I\sqcup I$ is also cofibrant, and hence $\Cyl(I)$ is cofibrant in $\Nu$, as $i$ is a cofibration. It follows from these observations that $F\otimes w$ is a pointwise weak equivalence in $[D,\ca]_0$. By Theorem \ref{cyl}, the Acyclicity condition is satisfied, and thus the injective model structure on $[D,\ca]_0$ exists. Moreover, it is $\Nu$-enriched copying the proof of Theorem \ref{enriched}.
\end{proof}

Theorem \ref{under} together with Remark \ref{cofib} yields examples of enriched model structures on symmetric spectra of simplicial sets and on chain complexes of modules over a commutative ring for which the injective or projective model structure on all categories of enriched diagrams exists. 

\begin{ex}[Spectra] \label{Sp}
Take $\Nu=\Sp$, the category of symmetric spectra of simplicial sets. This is a locally presentable base with the monoidal structure given by the smash product $\wedge$ for the tensor product and the sphere spectrum $\mathbb{S}$ for the unit. By \cite[III.4.13]{Sch}, the stable injective model structure on $\mathrm{Sp}^\Sigma$ is combinatorial, and every object is cofibrant in this model structure. It is enriched over the projective model structure on~$\mathrm{Sp}^\Sigma$, as shown in \cite[Theorem 5.3.7, parts 3 and 5]{HSS}.  Moreover, the sphere spectrum~$\mathbb{S}$ is cofibrant in the projective model structure. Hence, by Theorem \ref{under} (i), the injective model structure on the category of spectral diagrams $[D,(\Sp)^{\inj}]$ exists and is enriched over $(\Sp)^{\proj}$, for every small spectral category $D$. Furthermore, the projective stable model structure on $\Sp$, given in \cite[III.4.11]{Sch}, is also combinatorial, and is enriched over itself, by \cite[Theorem 5.3.7, parts 1 and 5]{HSS}. In this model structure, all objects are fibrant. Hence, by Theorem \ref{under} (ii), the projective model structure on~$[D,(\Sp)^{\proj}]$ exists and is enriched over $(\Sp)^{\proj}$, for every small spectral category $D$.
\end{ex}

\begin{ex}[Chain complexes] \label{ChR}
Given a commutative ring $R$, take $\Nu=\ChR$, the category of chain complexes of $R$-modules. This is a locally presentable base with the monoidal structure given by the tensor product $\otimes_R$ and the ring $R$ concentrated in degree~$0$ for the unit. The injective model structure on $\ChR$, introduced by Hovey in \cite[Theorem 2.3.13]{Hov}, is combinatorial, and every object is cofibrant in this model structure. By \cite[Proposition 3.3]{Shi}, it is enriched over the projective model structure on $\ChR$. Moreover, the ring~$R$ is cofibrant in the projective model structure. Hence, by Theorem~\ref{under}~(i), the injective model structure on the category of dg functors~$[D,\ChR^{\inj}]$ exists and is enriched over~$\ChR^{\proj}$, for every small dg category $D$. Furthermore, the projective model structure on $\ChR$, due to Quillen \cite{Qui}, is also combinatorial, by \cite[Theorem 2.3.11]{Hov}, and is enriched over itself, by \cite[Theorem 1.4]{BMR}. In this model structure, all objects are fibrant. Hence, by Theorem \ref{under} (ii), the projective model structure on $[D,\ChR^{\proj}]$ exists and is enriched over $\ChR^{\proj}$, for every small dg category~$D$. Dg functors $D\to \ChR$ actually corresponds to dg $D$-modules, and we recover the projective and injective model structures on the category of dg $D$-modules established in \cite[Theorem 3.2]{Keller} with these two examples. 
\end{ex}

\begin{ex}[Hurewicz model structure] \label{Hur}
For $R$ a commutative ring, there is another model structure on $\ChR$: the Hurewicz model structure, due to Christensen and Hovey in \cite{CH}, Cole in \cite{Cole}, and Schw\"anzl and Vogt in \cite{SV}. Christensen and Hovey show that this model structure is not cofibrantly generated in \cite[Section 5.4]{CH}. It is anyway accessible, by \cite[Section 6.4]{BMR} and \cite[Theorem 4.2.1]{HKRS}. Moreover, \cite[Theorem 1.15]{BMR} states that it is a model structure which is enriched over itself, and in which all objects are cofibrant and fibrant. Hence, by Theorem~\ref{under}, or also by Theorem \ref{enriched} and Remark \ref{bla}, the injective and projective model structures on $[D, \ChR^{\mathrm{Hur}}]$ exist and are enriched over $\ChR^{\mathrm{Hur}}$, for every small dg category~$D$. 
\end{ex}

\section{Application: modules over an operad} \label{sec7}

Let $(\Nu,\otimes, I)$ be a locally presentable base and a model category. Let $O$ be an operad in $\Nu$, and let $\ca$ be a $\Nu$-enriched model category. Adapting Theorems~\ref{enriched} and~\ref{under}, we can give criteria for the injective and projective model structures on $\Mod_\ca(O)$ to exist, where $\Mod_\ca(O)$ denotes the category of right modules over the operad~$O$ with values in $\ca$. This is a direct consequence of a result by Arone and Turchin in \cite{AT} which characterizes such modules in terms of contravariant $\Nu$-functors into $\ca$ from a small $\Nu$-category $\cf(O)$, which we first describe here.

\begin{notation}
Let $S,T$ be two finite sets. We write $F(S,T)$ for the set of maps from~$S$ to~$T$.
\end{notation}

\begin{defn}
Let $O$ be an operad in $\Nu$. The $\Nu$-category $\cf (O)$ has finite sets as objects, and its hom-objects are defined to be 
\[ \begin{array}{lr}
\displaystyle \cf (O)(S,T) = \bigcup_{\alpha\in F(S,T)} \bigotimes_{t\in T} O(\alpha^{-1}(t))  & \in \Nu, 
\end{array}\]
for all finite sets $S,T$. See \cite[Definition 3.1]{AT} for more details.
\end{defn}

This gives rise to the following characterization of the right modules over $O$.

\begin{prop}[\!{\cite[Lemma 3.3]{AT}}] \label{mod}
Let $O$ be an operad in $\Nu$, and let $\ca$ be a tensored and cotensored $\Nu$-category. Then there is an equivalence of categories
\[ \Mod_\ca(O) \simeq [\cf(O)^{\op}, \ca]_0. \]
\end{prop}

We first adapt Theorem \ref{enriched} to this setting.

\begin{cor} \label{operad}
Let $(\Nu,\otimes, I)$ be a locally presentable base and a model category, and let~$O$ be an operad in $\Nu$. Suppose $\ca$ be a locally $\Nu$-presentable accessible $\Nu$-enriched model category. 
\begin{rome}
\item If the functors $O(S)\otimes -\colon \ca_0\to \ca_0$ preserve cofibrations for all $S\in \cf(O)$, there is an injective model structure on $\Mod_\ca(O)$, where the weak equivalences and cofibrations are defined pointwise, and this model structure is enriched over $\Nu$. 
\item If the functors $O(S)\otimes -\colon \ca_0\to \ca_0$ preserve trivial cofibrations for all $S\in \cf(O)$, there is a projective model structure on $\Mod_\ca(O)$, where the weak equivalences and fibrations are defined pointwise, and this model structure is enriched over~$\Nu$. 
\end{rome}
\end{cor}

\begin{rem}
In particular, if the operad $O$ is pointwise cofibrant, i.e.~if all objects~$O(S)$, for $S\in \cf(O)$, are cofibrant in $\Nu$, then the injective and projective model structures on~$\Mod_\ca(O)$ exist.
\end{rem}

\begin{ex}
As in Example \ref{sSet}, taking $\Nu=\sset_{(*)}$, the injective and projective model structures on $\Mod_\ca(O)$ exist and are simplicial, for every simplicial operad $O$ and every enriched locally presentable simplicial accessible model category $\ca$, e.g.~$\ca=\sset_{(*)}, \Sp$.
\end{ex}

We also give a version of Theorem \ref{under} in this setting. 

\begin{cor} \label{unmod}
Let $(\Nu,\otimes,I)$ be a locally presentable base and a model category, and let~$O$ be an operad in $\Nu$. Suppose $\ca$ be a locally $\Nu$-presentable accessible $\Nu$-enriched model category.
\begin{rome}
\item If the unit $I\in \Nu$ is cofibrant and the category $\Mod_\ca(O)$ admits underlying cofibrant replacements, e.g.~if all objects of $\ca$ are cofibrant, there is an injective model structure on $\Mod_\ca(O)$, which is enriched over $\Nu$.
\item If the unit $I\in \Nu$ is fibrant and the category $\Mod_\ca(O)$ admits underlying fibrant replacements, e.g.~if all objects of $\ca$ are fibrant, there is a projective model structure on $\Mod_\ca(O)$, which is enriched over $\Nu$.
\end{rome}
\end{cor}

The examples at the end of Section \ref{sec6} give rise to examples of categories of modules over an operad, which admits injective and projective model structures.

\begin{ex} \label{Spp}
As in Example \ref{Sp}, let $\Nu=(\Sp)^{\proj}$ be the category of symmetric spectra of simplicial sets with the projective stable model structure. Suppose $O$ is a spectral operad. If~$\ca=(\Sp)^{\inj}$ is endowed with the injective stable model structure, there is an injective model structure on $\Mod_{\Sp}(O)$, which is enriched over $(\Sp)^{\proj}$, by Corollary~\ref{unmod}~(i). On the other hand, if we take $\ca=\Nu=(\Sp)^{\proj}$, there is a projective model structure on $\Mod_{\Sp}(O)$, which is enriched over $(\Sp)^{\proj}$, by Corollary \ref{unmod} (ii). 
\end{ex}

\begin{ex}
The same reasoning as above applies to the setting of Example \ref{ChR}. Given a commutative ring $R$, let $\Nu=\ChR^{\proj}$ be the category of chain complexes of $R$-modules with the projective model structure. If $O$ is a dg operad, then there exist both injective and projective model structures on $\Mod_{\ChR}(O)$, and these are enriched over $\ChR^{\proj}$. 
\end{ex}

\begin{ex}
Given a commutative ring $R$, let $\Nu=\ca=\ChR^{\mathrm{Hur}}$ be the category of chain complexes of $R$-modules with the Hurewicz model structure. Using results of Example \ref{Hur}, if $O$ is a dg operad, then there exist both injective and projective model structures on $\Mod_{\ChR^{\mathrm{Hur}}}(O)$, and these are enriched over $\ChR^{\mathrm{Hur}}$, by Corollary \ref{operad} or \ref{unmod}.
\end{ex}

\section{Properness} \label{sec8}

In this last section, we prove that the properness of the initial model structure transfers to the injective and projective model structures on the category of enriched diagrams under mild additional assumptions to the ones of Theorem \ref{injproj}. Recall that a model category is \emph{left-proper} if weak equivalences are preserved under pushouts along cofibrations. Dually, it is \emph{right-proper} if weak equivalences are preserved under pullbacks along fibrations. We give the proofs for the injective case, while the proofs for the projective one are dual.

\begin{prop} \label{leftproper}
Let $(\Nu, \otimes, I)$ be a locally presentable base. Suppose $\ca$ is a locally $\Nu$-presentable $\Nu$-category, whose underlying category $\ca_0$ admits a left-proper accessible model structure. Let $D$ be a small $\Nu$-category.
\begin{rome}
\item If the functors $-\otimes D(d,d')\colon \ca_0\to \ca_0$ preserve cofibrations for all $d,d'\in D$, then the injective model structure on $[D,\ca]_0$ exists, and is also left-proper.
\item If the functors $-\otimes D(d,d')\colon \ca_0\to \ca_0$ preserve cofibrations and trivial cofibrations for all~$d,d'\in D$, the projective model structure on $[D, \ca]_0$ exists, and is also left-proper. 
\end{rome}
\end{prop}

\begin{proof}
Since colimits are computed pointwise in $[D,\ca]_0$, (i) follows directly from the fact that cofibrations and weak equivalences in the injective model structure are defined pointwise.
\end{proof}

\begin{prop} \label{rightproper}
Let $(\Nu, \otimes, I)$ be a locally presentable base. Suppose $\ca$ is a locally $\Nu$-presentable $\Nu$-category, whose underlying category $\ca_0$ admits a right-proper accessible model structure. Let $D$ be a small $\Nu$-category.
\begin{rome}
\item If the functors $-\otimes D(d,d')\colon \ca_0\to \ca_0$ preserve cofibrations and trivial cofibrations for all $d,d'\in D$, the injective model structure on $[D,\ca]_0$ exists, and is also right-proper.
\item If the functors $-\otimes D(d,d')\colon \ca_0\to \ca_0$ preserve trivial cofibrations for all ${d,d'\in D}$, the projective model structure on $[D, \ca]_0$ exists, and is also right-proper. 
\end{rome}
\end{prop}

\begin{proof}
Since limits are computed pointwise in $[D,\ca]_0$, it is enough to see that injective fibrations are in particular pointwise fibrations. This follows from the fact that the functors $-\otimes D(d,d')\colon \ca_0\to \ca_0$ preserve trivial cofibrations for all $d,d'\in D$. To see this, one can use a similar argument to the one used in the proof of Theorem \ref{injproj} to prove that injective trivial fibrations are in particular pointwise trivial fibrations when these tensor functors preserve cofibrations. 
\end{proof}

\bibliographystyle{alpha}

\end{document}